\documentclass{article}[12pt]

\usepackage{amssymb}
\usepackage{amsthm}
\usepackage{amsmath}
\usepackage[pdftex]{hyperref} 

\newtheorem{theorem}{Theorem}
\newtheorem{lemma}[theorem]{Lemma}
\newtheorem{corollary}[theorem]{Corollary}

\newtheorem{Prob}{Problem}[section]

\title{\bf On the Nonexistence of Some\\
Generalized Folkman Numbers\footnote{
Supported by the National Natural Science
Foundation (11361008) and the Guangxi Natural Science
Foundation (2011GXNSFA018142).}}

\author{Xiaodong Xu\\[-0.1ex]
\small Guangxi Academy of Sciences\\[-0.6ex]
\small Nanning 530007, P.R. China\\[-0.6ex]
\small {\tt xxdmaths@sina.com}\\[1.3ex]\and
Meilian Liang\\[-0.1ex]
\small School of Mathematics and Information Science\\[-0.6ex]
\small Guangxi University, Nanning 530004, P.R. China\\[-0.6ex]
\small {\tt gxulml@163.com}\\[1.3ex]\and
Stanis{\l}aw Radziszowski\\[-0.1ex]
\small Department of Computer Science\\[-0.6ex]
\small Rochester Institute of Technology, Rochester, NY 14623\\[-0.6ex]
\small {\tt spr@cs.rit.edu}\\[3.3ex]
}

\date{\today}

\begin{document}
\maketitle
\thispagestyle{empty}

\begin{abstract}
For an undirected simple graph $G$, we write
$G \rightarrow (H_1, H_2)^v$ if and only if for
every
red-blue coloring of its vertices there
exists a red $H_1$ or a blue $H_2$.
The generalized vertex Folkman number $F_v(H_1, H_2; H)$
is defined as the smallest integer $n$ for which there
exists an $H$-free graph $G$ of order $n$ such that
$G \rightarrow (H_1, H_2)^v$. The generalized
edge Folkman numbers
$F_e(H_1, H_2; H)$ are defined similarly, when colorings
of the edges are considered.

We show that
$F_e(K_{k+1},K_{k+1};K_{k+2}-e)$ and
$F_v(K_k,K_k;K_{k+1}-e)$ are well defined for $k \geq 3$.
We prove the nonexistence
of $F_e(K_3,K_3;H)$ for some $H$, in particular
for $H=B_3$, where $B_k$ is the book graph of
$k$ triangular pages, and for $H=K_1+P_4$.
We pose three problems on
generalized Folkman numbers, including the
existence question of edge Folkman numbers
$F_e(K_3, K_3; B_4)$,
$F_e(K_3, K_3; K_1+C_4)$ and
$F_e(K_3, K_3; \overline{P_2 \cup P_3} )$.
Our results lead to some general inequalities
involving two-color and multicolor Folkman numbers.
\end{abstract}

{\bf AMS classification subjects:} 05C55, 05C35

\section{Introduction}
\label{sintro}
Let $G$ be a finite undirected graph that contains no loops
or multiple edges. Denote by $V(G)$ the set of its vertices and
$E(G)$ the set of its edges.
For vertex-disjoint graphs
$G$ and $H$, the join graph $G+H$ has the set of vertices
$V(G) \cup V(H)$ and edges
$E(G) \cup E(H) \cup \{ \{(u,v\}\ |\ u \in V(G), v \in V(H)\}$.
For a set of vertices $S \subset V(G)$,
$G[S]$ is the graph induced by $S$ in $G$, and $G-u$ is the
graph obtained from $G$ by removing vertex $u \in V(G)$
together with all the edges adjacent to $u$.

The complete graph of order $n$ is
denoted by $K_n$, and a cycle of length $n$ by $C_n$.
The book graph $B_k$ is defined as $K_1 + K_ {1, k}$,
and the complete graph $K_n$ with one missing edge
will be denoted by $J_n$.
The clique number of $G$ will be denoted by $cl(G)$, and
the chromatic number of $G$ by $\chi(G)$. An $(s,t)$-graph
is a graph that does not contain $K_s$ neither any independent
sets of $t$ vertices. The set $\{1, \cdots, n\}$ will be
denoted by $[n]$.

\smallskip
For graph $G$, we write $G \rightarrow (H_1,H_2)^v$
if and only if for every
red-blue coloring $\chi$ of the vertices $V(G)$ there
exists a red subgraph $H_1$ or a blue subgraph $H_2$
in $\chi$.
The generalized vertex Folkman number $F_v(H_1,H_2; H)$
is defined as the smallest integer $n$ for which there
exists an $H$-free graph $G$ of order $n$ such that
$G \rightarrow (H_1, H_2)^v$. The set of all $H$-free
graphs satisfying the latter vertex arrowing will
be denoted by $\mathcal{F}_v(H_1, H_2; H)$.

The generalized edge Folkman numbers
$F_e(H_1, H_2; H)$ are defined similarly, when colorings
of the edges are considered. We write
$G \rightarrow (H_1, H_2)^e$ if and only if for every
red-blue coloring $\chi$ of the edges $E(G)$ there
exists a red subgraph $H_1$ or a blue subgraph $H_2$
in $\chi$.
The generalized edge Folkman number $F_e(H_1, H_2; H)$
is defined as the smallest integer $n$ for which there
exists an $H$-free graph $G$ of order $n$ such that
$G \rightarrow (H_1, H_2)^e$. The set of all $H$-free
graphs satisfying the latter edge arrowing
will be denoted by $\mathcal{F}_e(H_1, H_2; H)$.

The cases when $H_1$, $H_2$ and $H$ are complete graphs
have been studied by many authors, for two and more colors,
in particular in
\cite{bikovfe334,
DudekRodl2010jctb,ADudek2010,
DudekRodl2011,Folkman,Nkonev2008,NN,Nenov2001,
Nenov2009a,NesetrilRodl,NesetrilRodl1981a,Fv445,XuShao2010a}.
Often, if the graphs $H_i$ and $H$ are complete,
we will simply write the order of the graph,
say, as in $F_e(s,t;k)$ instead of $F_e(K_s,K_t;K_k)$.
In this paper we focus on two colors, but we will also
make some comments related to more colors, such as in
commonly studied multicolor vertex Folkman numbers
$F_v(a_1, a_2, \cdots, a_r;s)$ and edge Folkman numbers
$F_e(a_1, a_2, \cdots a_r;s)$, where $a_i$'s are the orders
of the arrowed complete graphs while coloring $K_s$-free
graphs. We note that the classical Ramsey number
$R(a_1, \cdots, a_r)$ can be defined as the smallest
integer $n$ such that $K_n \rightarrow (a_1, \cdots, a_r)^e$.
In the diagonal case
$a_1 = \cdots = a_r =a$ we may use a more
compact notation
$F_v^r(a;s) = F_v(a_1 ,\cdots ,a_r;s)$ and
$\mathcal{F}_v^r(a;s) = \mathcal{F}_v(a_1 ,\cdots ,a_r; s)$,
similarly
$F_e^r(a;s) = F_e(a_1 ,\cdots ,a_r;s)$ and
$\mathcal{F}_e^r(a;s) = \mathcal{F}_e(a_1 ,\cdots ,a_r; s)$,
as well as for arrowing general graphs, such as
in $F_e^r(G;H)$.

\smallskip
In 1970,
Folkman \cite{Folkman} proved that for any
integer $s > \max \{a_1, \cdots, a_r\}$, both sets
$\mathcal{F}_v(a_1, \cdots, a_r; s)$ and $\mathcal{F}_e(a_1,a_2;s)$
are nonempty, and thus the corresponding Folkman numbers are
well defined. In 1976,
Ne\v{s}et\v{r}il and R\"{o}dl \cite{NesetrilRodl}
generalized this result to the multicolor edge cases,
namely they proved that the sets
$\mathcal{F}_e(a_1, \cdots, a_r; s)$ are also nonempty,
for arbitrary $r \geq 2$ and $s > \max \{a_1, \cdots, a_r\}$.
An interesting upper bound on $F_v^r(a;s)$ was obtained by
Dudek and R\"{o}dl \cite{DudekRodl2010jctb} in 2010, as
in the following theorem.

\medskip
\begin{theorem} \label{pupperbound}
{\rm \cite{DudekRodl2010jctb}}
For any positive integer $r$ there exists a
constant $C = C(r)$ such that for every
$s \geq 2$ it holds that
$F_v^r(s;s+1) \leq Cs^2 \log{^4}s$.
\end{theorem}

\smallskip
The above determines that both vertex and edge Folkman
numbers exist when the arrowed and avoided graphs are complete,
for $s > \max \{a_1, \cdots, a_r\}$.
By simple monotonicity,
this easily extends to some cases (say, when the arrowed
graphs $H_i$ have at most $a_i$ vertices), but apparently
it poses interesting existence questions in other cases.
Only some special parameters are discussed in the literature,
such as the bound $F_e(K_4-e,K_4-e;K_4) \le 30193$
obtained by Lu \cite{Lu} in 2008. In this paper we focus
on some general situations, in particular when the avoided
graph $H$ is not the complete graph $K_s$, but
$H_1$ and $H_2$ are complete, and often $H_1=H_2=K_3$.

The nonexistence of a Folkman
number with some parameters is equivalent to the emptiness
of the corresponding set of Folkman graphs. For example,
the Folkman number $F_e(K_3,K_3;K_1+P_4)$ does not
exist if and only if $\mathcal{F}_e(K_3,K_3;K_1+P_4)=\emptyset$,
which in fact we prove to be true in Theorem
\ref{Fe33k1+p4nonexist}, Section 4.

The summary of contents of the remainder of
this paper is as follows:
Vertex and edge arrowing by $(K_s-e)$-free graphs
and related existence questions are discussed in
Section 2, similarly for graphs involving book
graphs in Section 3. Other cases involving
wheels and paths are analyzed in Section 4.
Finally, in Section 5 some results for
more than two colors are presented.

\medskip
\section{Ramsey arrowing by $(K_n-e)$-free graphs}\label{Kk+1-e}

Recall from the introduction that $J_k = K_k - e$.
One can easily see that $F_v(2,2;3) = F_v(K_2,K_2;K_3) = 5$,
which can be equivalently stated as that the smallest number
of vertices in any triangle-free graph $G$ with $\chi(G)>2$
is equal to 5. However, it is also easy to observe that
$F_v(K_2,K_2;J_3)$ does not exist, since every $J_3$-free
graph is bipartite. Similarly, we see that $F_e(K_3, K_3; J_4)$
does not exist, since in any $J_4$-free graph no two
triangles can share an edge, and thus
the edges of every triangle
can be independently red-blue colored. These
observations lead to our first theorem.

\begin{theorem} \label{FeFvKk-e}
For $k \geq 3$, if the edge Folkman number
$F_e(K_{k+1},K_{k+1}; J_{k+2})$ exists,
then the vertex Folkman number
$F_v(K_k,K_k;J_{k+1})$ exists too.
\end{theorem}

\begin{proof}
Suppose that $F_e(K_{k+1},K_{k+1};J_{k+2})$ exists,
it is equal to $n$, and let $G$ be any graph of order $n$
in $\mathcal{F}_e(K_{k+1},K_{k+1};J_{k+2})$.
For any vertex $u \in V(G)$ we must have
$G-u \not\rightarrow (K_{k+1},K_{k+1})^e$.
Fix any vertex $u \in V(G)$, and let $H$ be the graph
induced in $G$ by the neighbors of $u$, $H=G[N(u)]$.
Clearly, $H$ is a $J_{k+1}$-free graph.

For contradiction, assume that $F_v(K_k,K_k;J_{k+1})$
does not exist. This implies that
$H \not\rightarrow (k,k)^v$, and hence there exists
a partition of $N(u)$ into $U_1 \cup U_2$ such that
both $G[U_1]$ and $G[U_2]$ are $K_k$-free.
Next, observe that any red-blue edge coloring witnessing
$G-u \not\rightarrow (K_{k+1},K_ {k+1})^e$
can be extended to whole $E(G)$,
without creating any monochromatic $K_{k+1}$,
by coloring the edges $\{\{u,v\} \in E(G) \; |\; v \in U_1\}$
red and coloring the edges
$\{\{u,v\} \in E(G) \;|\; v \in U_2\}$ blue.
This contradicts that 
$G \in \mathcal{F}_e(K_{k+1},K_{k+1};J_{k+2})$,
and completes the proof.
\end{proof}

\medskip
Graph $H$ is called a Ramsey graph
for $K_n$ if $H \rightarrow (K_n, K_n)^e$.
In 1981, Ne\v{s}et\v{r}il and R\"{o}dl \cite{NesetrilRodl1981a}
proved the following theorem.

\smallskip
\begin{theorem} \label{NesetrilRodl1981}
{\rm \cite{NesetrilRodl1981a}}
Let $n \geq 3$ be a fixed positive integer.
Then there exists a Ramsey graph $H$ for $K_n$
such that any two subgraphs $K$, $K'$ of $H$
isomorphic to $K_n$ intersect in at most two points.
\end{theorem}

\begin{corollary} \label{Fexistence}
For every integer $k \geq 3$,

\noindent
{\rm (a)}
the edge Folkman number
$F_e(K_{k+1},K_{k+1};J_{k+2})$ exists, and

\noindent
{\rm (b)}
the vertex Folkman number
$F_v(K_k,K_k;J_{k+1})$ exists.
\end{corollary}

\begin{proof}
Graph $H$ in Theorem \ref{NesetrilRodl1981} does
not contain $J_{n+1}$ for $n \ge 4$, thus
if $n=k+1$ then
the set $\mathcal{F}_e(K_{k+1},K_{k+1};J_{k+2})$
is nonempty, and hence part (a) of the corollary
follows.  Theorem \ref{FeFvKk-e} and part (a)
imply part (b).
\end{proof}

\bigskip
We can easily see that for integers $s$ and $t$,
if $k > s \geq t \geq 2$, then $F_v(K_s,K_t;J_{k+1})$
exists, and by monotonicity
$F_v(K_s,K_t;J_{k+1}) \leq F_v(s,t;k)$.
The upper bound for $F_e(K_{k+1},K_{k+1};J_{k+2})$
which can be obtained using the proof of
Theorem \ref{NesetrilRodl1981}
is large, and likely it is much larger than the exact value.
Similarly, the implied upper bound for $F_v(K_k,K_k;J_{k+1})$
is likely much larger than the exact value.
It would be interesting to obtain better upper bounds for
these numbers directly without using Theorem \ref{NesetrilRodl1981},
for example by a method
similar to one used in the proof of
Theorem 1 in \cite{DudekRodl2010jctb}.
We note that a straightforward reasoning similar
to a method used in \cite{Nkonev2008} leads to an inequality
$F_v(K_{s_1s_2},K_{t_1t_2};$ $J_{k_1k_2+1}) \leq
F_v(s_1,t_1;k_1+1)F_v(K_{s_2},K_{t_2};J_{k_2+1}),$ for
$2 \leq s_1 \leq t_1 \leq k_1$ and
$3 \leq s_2 \leq t_2 \leq k_2$.
This makes us anticipate that $F_v(K_k,K_k;J_{k+1})$ grows
slowly with $k$, and possibly can be bounded by
$c F_v(k,k;k+1)$ for some constant $c>0$.

The best known concrete lower and upper bounds on
various Ramsey numbers of the form $R(J_s,K_t)$
are collected in \cite{SRN};
for example, we know that $30 \leq R(J_5,K_5) \leq 33$.
In that case, any 29-vertex witness graph to Ramsey lower
bound seems to be a good candidate for the vertex Folkman
number case of arrowing $(3,4)^v$. This would give
an interesting bound $F_v(K_3,K_4;J_5)\le 29$
(unfortunately, we were not successful in finding any
such graph so far).
Still we think that, in general, further exploration
of witnesses to lower bounds for
Ramsey numbers as graphs
showing upper bounds for (vertex or edge)
Folkman numbers is worth an effort.

\section{Arrowing triangles by $B_k$-free graphs}
\label{Bk}

Recall that the book graph $B_k$ was defined as
$B_k=K_1+K_{1,k}$, hence it has $k+2$ vertices
and consists of $k$ triangles sharing one common
edge. In particular, $B_1=K_3$, $B_2=J_4$ and
$B_3=K_5 \setminus K_3$. Thus, the first book-specific
case (different from $K_k$ and $J_k$) is that for the
book graph $B_3$ considered in the next theorem.

\begin{theorem} \label{k4b3freearrows33v}
There exists a $B_3$-free and $K_4$-free graph
$G$ of order $19$ such that $G \rightarrow (3,3)^v$.
Thus we have $F_v(K_3,K_3;B_3)\le 19$.
\end{theorem}

\noindent
{\bf Note.}
For the upper bound in the second part of the theorem
it is not required that the graph $G$ is $K_4$-free.
In any case, we consider the bound in Theorem
\ref{k4b3freearrows33v} quite strong.
Finding the actual value of $F_v(K_3,K_3;B_3)$
can be difficult, and it is open
whether the best construction must contain $K_4$.

\begin{proof}
We will construct the required graph $G$
on the vertex set
$V(G)= \bigcup_ {i=0} ^3 V_i \cup \{u\}$,
where $G[V_0]=K_3$, and
the subgraphs induced by $V_i$ are
isomorphic to $C_5$, for $i \in \{1,2,3\}$.
Let $V_0 = \{v_1, v_2, v_3\}$.
The other edges of $G$ are all possible
edges between $u$ and $\bigcup_ {i=1} ^3 V_i$,
and all possible edges between $v_i$
and the vertices in $V_i$, for each $i \in \{1,2,3\}$.
Thus $G$ has 19 vertices and 48 edges.

It is easy to see that the graph $G$ is both
$K_4$-free and $B_3$-free. With little more
effort, one can show that every red-blue
coloring of $V(G)$ contains a monochromatic
triangle. Without loss of generality we can assume
that $u$ is red and at least one of the vertices
in $V_0$, say $v_1$, is blue. However, in order
to avoid a red triangle on $u$ and two vertices
in $V_1$, $G[V_1]$ must contain a blue $K_2$. But
the latter together with $v_1$ would form a blue
triangle. Hence we have that
$G \rightarrow (3,3)^v$. Finally, the same
graph $G$ is a witness of the upper bound.
\end{proof}

\bigskip
Since $B_2=J_4$, and using the observation
from the beginning of Section 2, we see that
$F_e(K_3,K_3;B_2)$ does not exist.
Now we will consider the existence of
$F_e(K_3,K_3;B_k)$ for $k \geq 3$,
starting with the case of $B_3$.

\medskip
\begin{theorem} \label{FeK3K3K5-K3}
The edge Folkman number
$F_e(K_3,K_3;B_3)$ does not exist.
\end{theorem}

\begin{proof}
Suppose that $F_e(K_3,K_3;B_3)$ exists,
it is equal to $n$, and let $G$ be any graph of order $n$
in $\mathcal{F}_e(K_3,K_3;B_3)$.
For any vertex $u \in V(G)$ we must have
$G-u \not\rightarrow (K_3,K_3)^e$.
Fix any vertex $u \in V(G)$, and let $H$ be the graph
induced in $G$ by the neighbors of $u$, $H=G[N(u)]$.
Since $G$ is $B_3$-free, $H$ does not contain $K_{1,3}$,
or equivalently has maximum degree at most 2.
Therefore any connected component of $H$ is
bipartite or it is an odd cycle.

We will show that any red-blue coloring $\chi$
of the edges of $G-u$, such that $\chi$ is without
monochromatic triangles, can be extended to $G$
without creating any monochromatic triangles.
This will contradict the definition of $G$ and thus it
will complete the proof.

For the edges
$\{u,v\}$, where $v$ is in a bipartite component of $H$,
we assign the color red or blue according to which part
of the bipartition $v$ belongs to. For vertices $v$ on
odd cycles in $H$, we proceed as follows. Let $U$ be the
vertex set of some odd cycle in $H$. We can partition
$U$ into $U_1 \cup U_2$ so that $H[U_1]$ has exactly
one edge, say $e$, and $U_2$ is an independent set in $H$.
If $\chi(e)$ is red (blue), then we color the edges in
$\{\{u,v\} \;|\; v \in U_1 \}$ blue (red), and the
edges in $\{\{u,v\} \;|\; v \in U_2\}$ red (blue).
\end{proof}

\bigskip
We were not able to answer the question whether
$F_e(K_3, K_3; B_4)$ exists, and hence we leave it
as an open problem for the readers. Note that for
every $k \geq 5$, the edge Folkman number
$F_e(K_3,K_3;B_k)$ exists, and it is equal to 6,
because the complete graph $K_6$ is $B_k$-free
and $K_6 \rightarrow (K_3, K_3)^e$.

\medskip
\begin{Prob} \label{FeK3K3B4}
Does the edge Folkman number $F_e(K_3,K_3;B_4)$ exist?
\end{Prob}

In Theorem \ref{k4b3freearrows33v} we constructed a
$K_4$-free and $B_3$-free graph $G$ vertex arrowing $(3,3)^v$.
We think that it is an interesting challenge to solve
the following graph existence problem for $K_4$-free
and book-free graphs edge arrowing $(3,3)^e$.

\begin{Prob} \label{FeK3K3Bk}
For which $k \ge 4$ there exists a
$K_4$-free and $B_k$-free graph $G$ such that
$G \rightarrow (3,3)^e$?
\end{Prob}

The answer seems not easy even just for $k=4$.
Note that a YES solution to Problem \ref{FeK3K3B4}
does not provide an answer to Problem \ref{FeK3K3Bk}
with $k=4$, while a NO answer to Problem \ref{FeK3K3B4}
implies a NO answer to Problem \ref{FeK3K3Bk} for $k=4$.
For Problem \ref{FeK3K3Bk},
we know that the answer is NO for $k=3$ by Theorem
\ref{FeK3K3K5-K3} (hence we ask only about cases
for $k \ge 4$), and clearly a YES answer for any $k$
would imply YES answers for all $t>k$.

\medskip
One of the most wanted Folkman numbers is
$F_e(3,3;4)=F_e(K_3,K_3;K_4)$, for which the currently
best known bounds are $20 \le F_e(3,3;4)$ \cite{bikovfe334}
and $F_e(3,3;4)$ $\le 786$ \cite{Lange2012}. The value of
$F_e(3,3;4)$ can be equivalently defined as the smallest
number of vertices in any $K_4$-free graph which is not
a union of two triangle-free graphs. An overview of what
is known about this problem was presented in \cite{sur14}.
In particular, it was conjectured by Exoo that a special
cubic residues $(4,12)$-graph $G_{127}$ on the
vertex set $\mathcal{Z}_{127}$ is a witness to a much
improved upper bound $F_e(3,3;4) \le 127$, and likely
its subgraphs may even give $F_e(3,3;4) \le 94$
(see \cite{sur14}).
The graph $G_{127}$ is $K_4$-free, has independence
number 11, is $B_{12}$-free, but it contains a large
number of subgraphs isomorphic to $B_{11}$.
The Exoo's conjecture can be stated as
$G_{127} \rightarrow (3,3)^e$.
If true, then it would give a YES
answer in Problem \ref{FeK3K3Bk} for all $k \ge 12$,
leaving open the cases for $4 \le k \le 11$.
Recall that by Theorem \ref{FeK3K3K5-K3} the
answer for $k=3$ is NO.

\medskip
\section{More on arrowing triangles}

In this section we study the existence of
$F_e(K_3,K_3;H)$ for connected graphs $H$.
First,
we observe that, since graph avoidance is
monotonic with respect to subgraphs, if
a graph $H$ is connected and $cl(H)\ge 4$,
then there exist $H$-free graphs edge arrowing
$(3,3)^e$, i.e. $F_e(K_3,K_3;H)$ exists, and
obviously $F_e(K_3,K_3;H) \leq F_e(3,3;4)$.
For 5 vertices, there are 4 such graphs, namely
$\widehat{K}_{4,i}$ for $i \in [4]$,
where $\widehat{K}_{n,s}$ is the graph obtained
by connecting a new vertex $v$ to $s$ vertices of a $K_n$.
Clearly,
the numbers $F_e(K_3,K_3;\widehat{K}_{4,i})$ exist for
$i \in [4]$, and
$F_e(K_3,K_3;\widehat{K}_{4,i+1}) \le F_e(K_3,K_3;\widehat{K}_{4,i})$
for $1 \le i \le 3$. In particular, note that
$\widehat{K}_{4,3}=J_5$,
$\widehat{K}_{4,4}=K_5$, and
we have the easy bounds $15=F_e(3,3;5) \le F_e(K_3,K_3;J_5)
\le F_e(3,3;4) \le 786$, using only what
is known about $F_e(3,3;k)$ \cite{sur14}.
For $\widehat{K}_{4,i}$-free graphs, $i=1,2$, we have
the following lemma.

\medskip
\begin{lemma}
$F_e(K_3,K_3;\widehat{K}_{4,2})=
F_e(K_3,K_3;\widehat{K}_{4,1})=F_e(3,3;4).$
\end{lemma}

\begin{proof}
By the monotonicity of $F_e(K_3,K_3;\widehat{K}_{4,i})$
mentioned above, it is sufficient to prove that
$F_e(K_3,K_3;\widehat{K}_{4,2}) \ge F_e(3,3;4)$.
We will show that for any graph
$G \in \mathcal{F}_e(K_3,K_3;\widehat{K}_{4,2})$
there exists a subgraph $G' \in \mathcal{F}_e(3,3;4)$ of $G$,
which will complete the proof. Define graph $G'$ on the same
set of vertices as $G$, with the set of edges
$E(G')=E(G) \setminus \{e\;|\; e \in K_4 \subset G\}$.
Obviously, $G'$ is $K_4$-free.
Since $G$ is $\widehat{K}_{4,2}$-free,
we can see that every triangle
in $G$ which is not a triangle in $G'$ has its three
vertices in the same $K_4$ of $G$.
Thus, any red-blue edge coloring of $E(G')$ without
monochromatic triangles can be extended to whole $E(G)$
by independently red-blue coloring the edges of each $K_4$.
This contradicts that
$G \in \mathcal{F}_e(K_3,K_3;\widehat{K}_{4,2})$.
Thus, no such coloring of $E(G')$ exists, and hence
$G' \in \mathcal{F}_e(3,3;4)$.
\end{proof}

\bigskip
In the remainder of this section, we will consider only
connected graphs $H$ with $K_3$ but without $K_4$.
There are three such graphs on 4 vertices,
namely $J_4$ and its subgraphs, and hence as
commented in Section 2, $F_e(K_3,K_3;H)$ does
not exist in these cases.
In the following, we focus attention on connected
graphs $H$ of order 5 with $cl(H)=3$, and leave
the study of such graphs with more than 5 vertices
for future work. The next theorem claims
the nonexistence of $F_e(K_3,K_3;H)$ for
a special 5-vertex graph $H=K_1+P_4$.

\medskip
\begin{theorem} \label{Fe33k1+p4nonexist}
The edge Folkman number
$F_e(K_3,K_3;K_1+P_4)$ does not exist.
\end{theorem}

\begin{proof}
The proof is very similar to that of Theorem \ref{FeK3K3K5-K3}.
Suppose contrary, that $F_e(K_3,K_3;K_1+P_4)$ exists,
it is equal to $n$, and let $G$ be any graph of order $n$
in $\mathcal{F}_e(K_3,K_3;K_1+P_4)$.
For any vertex $u \in V(G)$ we must have
$G-u \not\rightarrow (K_3,K_3)^e$.
Fix any vertex $u \in V(G)$, and let $H$ be the graph
induced in $G$ by the neighbors of $u$, $H=G[N(u)]$.
Since $G$ is $(K_1+P_4)$-free, $H$ does not contain $P_4$.
Therefore any connected component of $H$ is
bipartite or isomorphic to $K_3$.
Now, the same steps as in the proof of
Theorem \ref{FeK3K3K5-K3} lead to a contradiction.
\end{proof}

\bigskip
We now state a theorem summarizing the existence of
$F_e(K_3,K_3;H)$ for all connected graphs $H$ on
5 vertices with $cl(H)=3$. Only two cases remain open,
namely those for the wheel graph $W_5$ and the
complement of $P_2 \cup P_3$. These cases should be
studied more, and we expect that new insights
can be important for better understanding
of which graphs edge arrow $(3,3)^e$.

\medskip
\begin{theorem} \label{Fen5}
Let $H$ be any connected $K_4$-free graph on
$5$-vertices containing $K_3$. Then the edge
Folkman number $F_e(K_3,K_3;H)$ does not exist,
except for two possible cases for $H$, namely $W_5$
and $\overline{P_2 \cup P_3}$.
\end{theorem}

\begin{proof}
There are 11 nonisomorphic $K_4$-free connected
graphs on 5 vertices containing $K_3$.
By Theorem \ref{Fe33k1+p4nonexist},
$F_e(K_3,K_3;K_1+P_4)$ does not exist. 
The graph $K_1+P_4$ contains as a subgraph 7 further
such graphs $H$
(including the bowtie graph $K_1+2K_2$, $K_{1,4}+ e$,
and the so-called bull graph),
for which by monotonicity $F_e(K_3,K_3;H)$ does not exist either.
This leaves three cases: $B_3$, $W_5$ and $\overline{P_2 \cup P_3}$.
The first case was eliminated by Theorem \ref{FeK3K3K5-K3},
while the other two are as the stated exceptions.
\end{proof}

\medskip
\begin{Prob} \label{Fek3k3W4}
Prove or disprove the existence

\smallskip
\noindent
{\rm (a)}
of the edge Folkman number $F_e(K_3,K_3;\overline{P_2 \cup P_3} )$, and

\noindent
{\rm (b)}
of the edge Folkman number $F_e(K_3,K_3;K_1+C_4)$.
\end{Prob}

\medskip
Note that $W_5=K_1+C_4$ is a subgraph of $J_5=K_5-e$.
Hence, if $F_e(K_3,K_3;W_5)$ exists, then we
have $F_e(K_3,K_3;J_5) \leq F_e(K_3,K_3;W_5)$.
The analogous statement holds for the complement
of $P_2 \cup P_3$. On the other hand, the latter
is a subgraph of $W_5$, hence there are only three
possible combined YES/NO answers to the existence
questions (a) and (b) in
Problem \ref{Fek3k3W4}, namely NO/NO, YES/YES and
NO/YES.

\smallskip
A natural direction to generalize considerations
of this section is to analyze which small graphs
on at least $6$ vertices necessarily are subgraphs
of every $K_4$-free graph edge arrowing $(3,3)^e$.
The simplest candidate for such a graph is
$B_4=K_6 \setminus K_4$, as stated in
Problem \ref{FeK3K3B4}. One could also proceed
by making a catalog of small subgraphs
in known witnesses of existence of $F_e(3,3;4)$,
in particular for the graph $G_{786}$,
which currently is the smallest known
such graph \cite{Lange2012}.
This, and even only some conditional answers
to our problems, may lead to better
bounds on $F_e(3,3;4)$.

\bigskip
\section{Some cases of multicolor Ramsey arrowing} \label{multicolor}

Since we know that $F_e(K_3,K_3;J_4)$ does not exist,
if a 3-color edge arrowing $G \rightarrow (K_3,K_3,K_k)^e$
holds, then we must have $G \rightarrow (J_4,K_k)^e$.
This easily generalizes to
$F_e(K_3,K_3,K_k;K_s) \geq F_e(J_4,K_k;K_s)$
for $s > k \geq 3$, and in particular it gives
$F_e(3,3,3;4) \geq F_e(J_4,K_3;K_4)$. We note that
$F_e(3,3,3;4)$ exists, its value
is unknown, it is likely quite large, and probably
still much harder to obtain than the notoriously
difficult case of $F_e(3,3;4)$.
Clearly, the same reasoning holds for any
graph $H$ instead of $J_4$ for
which $F_e(K_3,K_3;H)$ does
not exist, including $B_3$, $K_1+P_4$ or
other graphs discussed in the previous section.
This leads to the following corollary.

\medskip
\begin{corollary} \label{FeHKk}
If $H$ is any graph for which $F_e(K_3,K_3;H)$
does not exist, then for $s > k \geq 3$ we have
$$F_e(3,3,k;s) \geq F_e(H,K_k;K_s).$$
\end{corollary}

\begin{proof}
As in the comments above, we observe
that any $n$-vertex graph $G$ witnessing
the upper bound $F_e(3,3,k;s) \le n$
must also satisfy $G \rightarrow (H,K_k)^e$.
Thus we have $F_e(H,K_k;K_s) \le n$.
\end{proof}

\bigskip
It would be interesting to construct a $K_4$-free graph
$G$ such that $G \rightarrow (K_3,J_4)^e$
but $G \not\rightarrow (3,3,3)^e$.
This might be quite hard since it is difficult
to construct any $K_4$-free graph that arrows $(K_3,J_4)^e$,
and it would be another challenge to show that it does not
arrow $(3,3,3)^e$. Similarly, obtaining any nontrivial
lower bound for the difference
$F_e(3,3,3;4) - F_e(K_3,J_4;K_4)$
seems difficult.

On the other hand, there exists an
interesting example of a $K_4$-free graph $G$
on 30193 vertices, constructed by Lu \cite{Lu}, such that
$G \rightarrow (J_4,J_4)^e$ (thus also $G \rightarrow (K_3,J_4)^e$).
It is possible that for this graph we have
$G \rightarrow (3,3,3)^e$, however we do not
know how to prove or disprove the latter. Also, note
that by an argument as in the proof of Corollary \ref{FeHKk}
we have $F_e^4(3;4) = F_e(3,3,3,3;4) \ge F_e(J_4,J_4;K_4)$.

\bigskip
Finally, we establish a new link between some
two-color edge Folkman numbers and multicolor
vertex Folkman numbers. They generalize a result
obtained in \cite{XuShao2010a}.

\smallskip
\begin{lemma}
For $k \ge s \ge 2$ and graphs $G$ and $H$, if $G$ is $H$-free,
$H \subset K_{k+1}$, and $G \rightarrow (K_s,K_k)^e$, then
for every vertex $u \in V(G)$ and $s-1$ colors we have
$G-u \rightarrow (K_k,\cdots,K_k)^v$.
\end{lemma}

\begin{proof}
For a contradiction, suppose that for some graphs $G$ and
$H$ as specified in the lemma, and for some vertex $u \in V(G)$,
there exists a partition $V(G-u) =\bigcup_{i = 1} ^{s-1}V_i$,
such that the graphs $G[V_i]$ are $K_k$-free, for
every $i \in [s-1]$.

Now, we color red or blue all the edges in $E(G)$
as follows. All edges in each $G[V_i]$, for $i \in [s-1]$,
are colored blue. The edges in $G[N(u)]$ are also blue.
The edges between $u$ and $N(u)$ are red, and
all other edges in $E(G)$, which are necessarily between
different parts $V_i$, are also colored red. Note that any
red clique may have at most one vertex in each
of the parts $V_i$, and that there are no red triangles
passing through vertex $u$. Thus, this coloring has no
red $K_s$. No nontrivial blue clique contains vertex $u$,
and none of $G[V_i]$ contains blue $K_k$, hence any potential
blue $K_k$ on vertices $S$ must intersect different parts $V_i$.
However, if such $S$ exists, and
because of how the coloring was defined,
the set of vertices $S \cup \{u\}$ would form a $K_{k+1}$,
contrary to the assumption that $G$ is $H$-free.
\end{proof}

\medskip
\begin{corollary} \label{Fe3kFvkkgeneral}
For $2 \leq s \leq k$ and graph $H \subset K_{k+1}$,
if $F_e(K_s,K_k;H)$ exists, then
$F_v^{s-1}(K_k;H)$ also exists and
$F_e(K_s,K_k;H) \ge F_v^{s-1}(K_k;H)+1$.
\end{corollary}

\begin{proof}
Consider any graph $G$, such that
$G \in \mathcal{F}_e(K_s,K_k;H)$, of the least
possible order $F_e(K_s,K_k;H)$.
Then by Lemma 11, the graph $G-u$ is in the set
$\mathcal{F}_v^{s-1}(K_k;H)$ and it has one
vertex less than $G$.
This proves the inequality.
\end{proof}

\bigskip
The proofs of our last lemma and corollary use
a method similar to one applied in the proof of
$F_e(3,k;k+1) > F_v(k,k;k+1)$ in \cite{XuShao2010a}.
The latter is a special case of Corollary \ref{Fe3kFvkkgeneral}
with $s=3$ and $H=K_{k+1}$. Another interesting
instantiation of Corollary \ref{Fe3kFvkkgeneral}
is for $H=J_{k+1}$, for which
the existence question of corresponding
Folkman numbers was discussed in Section 2.

\end{document}